\documentclass[a4paper,12pt,oneside]{amsart}

\usepackage[german,english]{babel}

\usepackage{amsmath,amssymb,amsthm,enumerate,mathtools}

\usepackage{dsfont}
\usepackage{color}
\usepackage[utf8]{inputenc} 
\usepackage[automark,nouppercase]{scrpage2}
\usepackage{hyperref} 
\usepackage[normalem]{ulem} 
\usepackage{cancel}
\usepackage[numbers]{natbib}

\newcommand{\mE}{\mathbb{E}}

\newcommand{\mP}{\mathbb{P}}
\newcommand{\mR}{\mathbb{R}}

\newcommand{\cF}{\mathcal{F}}

\newcommand{\1}{\mathds{1}}
\newcommand\dd{\mathrm{d}}

\newtheorem {definition}{Definition}[section]
\newtheorem {theorem}[definition]{Theorem}
\newtheorem{proposition}[definition]{Proposition}
\newtheorem {lemma}[definition]{Lemma}
\newtheorem {corollary}[definition]{Corollary}

\title[O\MakeLowercase{ne-sided stable processes conditioned to avoid an interval}]{Completely asymmetric stable processes conditioned to avoid an interval}

\author{Pierre Lenthe}
\address{Pierre Lenthe: University of Mannheim, Institute of Mathematics, 68161 Mannheim, Germany.}
\email{plenthe@mail.uni-mannheim.de}

\author{Philip Weißmann$^{*}$}\thanks{$^{*}$Supported by the Research Training Group "Statistical Modeling of Complex Systems" funded by the German Science Foundation}
\address{Philip Weißmann: University of Mannheim, Institute of Mathematics, 68161 Mannheim, Germany.}
\email{hweissma@mail.uni-mannheim.de}

\allowdisplaybreaks

\begin{document}

\begin{abstract}
In the recent article Döring et al. \cite{Doer_Kyp_Wei_01} the authors conditioned a stable process with two-sided jumps to avoid an interval. As usual the strategy was to find an invariant function for the process killed on entering the interval and to show that the corresponding $h$-transformed process is indeed the process conditioned to avoid an interval in a meaningful way. In the present article we consider the case of a completely asymmetric stable process. It turns out that the invariant function found in \cite{Doer_Kyp_Wei_01} does not exist or is not invariant but nonetheless, we will characterize the conditioned process as a Markov process.
\end{abstract}

\maketitle 

\section{Introduction}
Conditioning Markov processes to avoid sets is a classical problem. Indeed, suppose $(\mP_x)_{x\in E}$
is a family of Markov probabilities on the state space $E$, and that  $T_B$ is the first hitting time of a fixed set $B$. When $T_B$ is almost surely finite, it is non-trivial to construct and characterise the conditioned process through the natural limiting procedure
\begin{align}\label{cond}
	&\lim_{s\to \infty}\mP_x(\Lambda\,|\, s <T_B) 
\end{align}
or the randomised version
\begin{align}\label{cond_exp}
	&\lim_{q \searrow 0}\mP_x(\Lambda, t < e_q \,|\, e_q<T_B),
\end{align}
for $\Lambda\in \mathcal F_t$ and $x\in E$. Here, $(\mathcal{F}_t)_{t\geq 0}$ denotes the natural enlargement of the filtration induced by the Markov process and $e_q$ are independent exponentially distributed random variables with parameter $q>0$. 

\smallskip

A classical example is the Brownian motion conditioned to avoid the negative half-line.
In this case, the limits \eqref{cond} and \eqref{cond_exp} lead to a so-called Doob $h$-transform
of the Brownian motion killed on entering the negative half-line, by the positive
invariant function $h(x) = x$ on $(0,\infty)$. This Doob $h$-transform
turns out (see Chapter VI.3 of \cite{Rev_Yor_01}) to be the Bessel process of dimension $3$,
which is transient. This example is typical, in that a conditioning procedure
leads to a new process which is transient where the original process was recurrent.
\smallskip

Extensions of this result have been obtained in several directions, most notably to random walks and L\'evy processes. A prominent example with several applications is that of a L\'evy process conditioned to stay positive, which was found by Chaumont and Doney \cite{Chau_Don_01} using
the randomised conditioning \eqref{cond_exp}. In that case, the associated invariant function $h$ is given by the potential function of the descending ladder height process.

\smallskip

In the recent articles \cite{Doer_Kyp_Wei_01} and \cite{Doer_Wat_Wei_01} the authors considered the problem of conditioning a Lévy process to avoid an interval. The first one is focused on stable processes, a subclass of Lévy processes which contains processes satisfying the scaling property
$$((cX_{c^{-\alpha}t})_{t \geq 0},\mP_x) \overset{(d)}{=} ((X_t)_{t \geq 0},\mP_{cx}), \quad x\in \mR, c>0,$$
where $\alpha \in (0,2)$ is the so called index of self-similarity. One assumption is that the stable processes have two-sided jumps. Our task is to fill the remaining gap and handle the case when the stable process has just one-sided jumps. We will show that the conditioned process in the sense of \eqref{cond_exp} is again a Markov process of which we will explicitly present the transition probabilities. Unfortunately, it is not possible to describe this Markov process as an $h$-transform of the process killed on entering the interval. Instead, it is a combination of the processes conditioned to stay above respectively below the interval. As it was done in \cite{Doer_Kyp_Wei_01} we restrict to the interval $[-1,1]$ because we can translate the process conditioned to avoid $[-1,1]$ to the process conditioned to avoid general intervals via the spatial homogeneity and the scaling property.

\smallskip

Before presenting our results, we introduce the most important definitions concerning Lévy processes, their subclass containing stable processes and Doob's $h$-transform. More details can be found, for example, in Bertoin \cite{Bert_01}, Kyprianou \cite{Kyp_01}, Sato \cite{Sat_01} or Chung and Walsh \cite{Chu_Wal_01}.\smallskip

\textbf{Lévy processes:} A Lévy process $X$ is a stochastic process with stationary and independent increments whose trajectories are almost surely right-continuous with left-limits.
For each $x \in \mathbb{R}$, we define the probability measure $\mP_x$ under which
the canonical process $X$ starts at $x$ almost surely. We write $\mP$ for the measure $\mP_0$. The dual measure $\hat{\mP}_x$ denotes the law of the so-called dual process $-X$ started at $x$.
A Lévy process can be identified using its characteristic exponent $\Psi$, defined by the equation
$\mE[ e^{\mathrm{i} q X_t}] = e^{-t\Psi(q)}$, $q\in \mR$, which has the Lévy-Khintchine representation:
\[ \Psi(q) = \mathrm i a q + \frac{1}{2} \sigma^2q^2
  + \int_{\mathbb{R}} (1-e^{\mathrm i q x} + \mathrm i q x \1_{\{\lvert x \rvert < 1\}} ) \, \Pi(\mathrm d x),\quad q \in \mR,
\]
where $a\in\mR$ is the so-called centre of process, $\sigma^2\geq 0$ is the variance of the Brownian component, and the Lévy measure $\Pi$ is a real measure with no atom at $0$ satisfying $\int_\mR (x^2\wedge 1)\, \Pi(\dd x)<\infty$.
\smallskip

\textbf{Killed L\'evy processes and $h$-transforms:}
For a Lévy process $X$ and a Borel set $B$ the killed transition measures are defined as
$$P^{B}_t(x,\dd y)=\mP_x(X_t \in \dd y, t< T_{B}), \quad t\geq 0.$$
The corresponding sub-Markov process is called the L\'evy process killed in $B$.
An invariant function for the killed process is a measurable function $h : \mR \backslash B \rightarrow [0,\infty)$ such that
\begin{align}\label{eq_def_harm}
	\mE_x \big[ \1_{\{ t < T_B \}} h(X_t) \big] = h(x),
	\quad x \in \mR \backslash B, t \geq 0.
\end{align}
An invariant function taking only strictly positive values is called a positive invariant function. Thanks to the Markov property, invariance is equivalent to $(\1_{\{ t < T_B \}} h(X_t))_{t \geq 0}$ being a $\mP_x$-martingale with respect to $(\cF_t)_{t \geq 0}$. When $h$ is a positive invariant function,  the associated Doob $h$-transform is defined via the change of measure
\begin{align}\label{def_htrafo}
	\mP^h_x(\Lambda) := \mE_x \Big[ \1_\Lambda \1_{\{ t < T_B \}} \frac{h(X_t)}{h(x)}  \Big],\quad x\in\mR \backslash B,
\end{align}
for $\Lambda \in \cF_t$. From Chapter 11 of Chung and Walsh \cite{Chu_Wal_01}, we know that under $\mP^h_{x}$ the canonical process is a conservative strong Markov process.

\smallskip

\textbf{Stable processes:} Stable processes are Lévy processes observing the scaling property
\begin{align}\label{scaling}
\left((c {X}_{c^{-\alpha}t})_{t \geq 0},\mP_x \right) \overset{(d)}{=} \left(({X}_{t})_{t \geq 0},\mP_{cx}\right)
\end{align}
for all $x \in \mR$ and $c>0$, where $\alpha$ is the index of self-similarity. It turns out to be necessarily the case that $\alpha\in (0,2]$ with $\alpha=2$ corresponding to the Brownian motion. The continuity of sample paths excludes the Brownian motion from our study, because conditioning to avoid the interval $[-1,1]$ translates to conditioning to stay above $1$ or below $-1$, depending on the initial value of the process. So we restrict to $\alpha\in (0,2)$.

\smallskip

As a Lévy process, stable processes are characterised entirely by their jump measure  which is known to be
\begin{align*}
	\Pi(\dd x) = \frac{\Gamma(\alpha+1)}{\pi} \left\{
\frac{ \sin(\pi \alpha \rho) }{ x^{\alpha+1}} \1_{\{x > 0\}} + \frac{\sin(\pi \alpha \hat\rho)}{ {|x|}^{\alpha+1} }\1_{\{x < 0\}}
 \right\}\dd x,\qquad  x \in \mathbb{R},\end{align*}
where 
$\rho:=\mP({X}_1 \geq 0)$ and $\hat{\rho}= 1-\rho$. We have taken a normalisation for which its characteristic exponent satisfies
\[
\mathbb{E}_x[{\rm e}^{{\rm i}\theta( {X}_1-x)}] = {\rm e}^{-|\theta|^\alpha}, \qquad \theta\in\mathbb{R}.
\]
Later, we will consider the special case of a stable process without negative jumps which forces $\rho=1$ when $\alpha<1$ (in which case $X$ is even a subordinator) and $\rho = 1-1/\alpha$ if $\alpha>1$. For $\alpha=1$ the stable process reduces to the symmetric Cauchy process with $\rho=\frac 1 2$. Since we are here just interested in one-sided jump processes we also exclude the symmetric Cauchy process from our study. 

An important fact we will use for the parameter regimes is that the stable process exhibits (set) transience and (set) recurrence according to whether $\alpha\in (0,1)$ or $\alpha\in[1,2]$. When $\alpha\in(1,2]$ the notion of recurrence is even stronger in the sense that individual points are hit with probability one.

\section{Main results}\label{sec:main}
Let from now on be $X$ a stable process with self-similarity index $\alpha \in (0,2)\setminus\{1 \}$ without negative jumps. Our results are separated between the parameter regimes $\alpha \in (0,1)$ and $\alpha \in (1,2)$.

\smallskip

If $\alpha<1$ the stable process is a subordinator. The following theorem presents an invariant function for the process killed on entering the interval.
\begin{theorem}\label{mainthm1}
Let $(X_t)_{t \geq 0}$ be a stable process with self-similarity index $\alpha \in (0,1)$ and no negative jumps. Then the function
$$h(x) =\begin{cases}
1 \quad &\text{if } x >1 \\
1 - \frac{\sin(\pi\alpha)}{\pi} \int_0^{\frac{2}{-1-x}}  t^{-\alpha} (1+t)^{-1}  \, \dd t & \text{if } x <-1
\end{cases}
$$
is invariant for the process killed on entering $[-1,1]$.
\end{theorem}
The key to prove this Theorem lies in showing $h(x) = \mP_x(T_{[-1,1]} = \infty)$. Using this invariant function we define the $h$-transform of the process killed on entering $[-1,1]$:
$$\mP_x^\updownarrow(\Lambda) := \mE_x\Big[ \1_\Lambda \1_{\{ t < T_{[-1,1]} \}}  \frac{h(X_t)}{h(x)} \Big], \quad \Lambda \in \cF_t.$$
The next result tells that the process conditioned to avoid $[-1,1]$ in the sense of \eqref{cond} equals this $h$-transform.
\begin{corollary}\label{cor}
Let $(X_t)_{t \geq 0}$ be a stable process with self-similarity index $\alpha \in (0,1)$ without negative jumps. Then it holds:
$$\mP_x^\updownarrow(\Lambda) =  \lim_{s \rightarrow \infty}  \mP_x(\Lambda \,|\, T_{[-1,1]} >s).$$
\end{corollary}

\smallskip

Since the case $\alpha<1$ is handled, we focus on $\alpha>1$ in which case the process is recurrent which implies in particular that $\mP_x(T_{[-1,1]}= \infty) = 0$ and hence, the methods used in the case $\alpha<1$ do not work.
\begin{proposition}\label{propmain}
Let $(X_t)_{t \geq 0}$ be a stable process with self-similarity index $\alpha \in (1,2)$ without negative jumps and
\begin{align}\label{transform}
\mP_x^\updownarrow(\Lambda) := \begin{cases}
\mP_x^\uparrow(\Lambda) := \mE_x \big[\1_\Lambda \1_{\{t < T_{(-\infty,1]}\}} \frac{X_t-1}{x-1}\big] \quad &\text{if } x>1\\
\mP_x^\downarrow(\Lambda) := \mE_x \big[\1_\Lambda \1_{\{t < T_{[-1,\infty)]}\}} \frac{(-1-X_t)^{\alpha-1}}{(-1-x)^{\alpha-1}} \big] &\text{if } x<-1
\end{cases},
\end{align}
for $\Lambda \in \cF_t$. Then $\mP_x^\updownarrow$ defines a probability measure under which $(X_t)_{t \geq 0}$ is a Markov process with transition probabilities
$$P_t^\updownarrow(x, \dd y) = \begin{cases}
\mE_x \big[ \1_{\{X_t \in \dd y\}} \1_{\{t < T_{(-\infty,1]}\}} \frac{X_t-1}{x-1}\big] \quad &\text{if } x,y>1\\
\mE_x \big[ \1_{\{X_t \in \dd y\}} \1_{\{t < T_{[-1,\infty)]}\}} \frac{(-1-X_t)^{\alpha-1}}{(-1-x)^{\alpha-1}} \big] &\text{if } x,y<-1 \\
0 & \text{else}
\end{cases}.$$
\end{proposition}
We will see that $X$ under $\mP_x^\updownarrow$ is just the stable process conditioned to stay above $1$ when it starts above $1$ and the stable process conditioned to stay below $-1$ when it starts below $-1$.

\smallskip

Now we are ready to present our main result which says that the process conditioned to avoid $[-1,1]$ in the sense of \eqref{cond_exp} corresponds to $(X,\mP_x^\updownarrow)$.
\begin{theorem}\label{mainthm}
Let $(X_t)_{t \geq 0}$ be a stable process with self-similarity index $\alpha \in (1,2)$ and no negative jumps. Then it holds
\begin{align}\label{eq_main}
\lim_{q \searrow 0} \mP_x(\Lambda, t < e_q \,|\, e_q < T_{[-1,1]}) = \mP^\updownarrow_x(\Lambda),
\end{align}
for $\Lambda \in \cF_t$ and $x \notin [-1,1]$.
\end{theorem}

When the starting value is above the interval $[-1,1]$ this result is not very surprising because the process conditioned to avoid the interval must be the same as the process conditioned to stay above $1$ since there are no negative jumps. But if the initial value is below $-1$ the result is rather surprising. In this case our statement says again that the process conditioned to avoid the interval is just the process conditioned to stay below $-1$. In other words, the case in which the process makes a jump from below to above the interval is cancelled by the limiting procedure.

\smallskip

A simple consequence is that the conditioned process converges a.s. to $+\infty$ if the initial value is above the interval and to $-\infty$ if the initial value is below the interval. This can be seen by the long-time behaviour of Lévy processes conditioned to stay positive, Chaumont and Doney \cite{Chau_Don_01}.

\smallskip

We remark that the natural procedure of plugging in the values of $\rho$ in the one-sided jumps case into the invariant function of the two-sided jumps case (found in \cite{Doer_Kyp_Wei_01}) does not work. When $\alpha<1$ (and $\rho = 1$) the function does not exist and when $\alpha>1$ (and $\rho = 1-1/\alpha$) one can show that the function is not invariant any more.

\section{Preliminaries}\label{sec::prel}
To prove our results we have to introduce some fluctuation theory for Lévy processes and Lévy processes conditioned to stay positive. For this section let $X=(X_t)_{t \geq 0}$ be a Lévy process.

\smallskip

\subsection{Ladder height processes and potential functions} A crucial ingredient in our analysis are the potential functions $U_+$ and $U_-$ of the ascending and descending ladder height processes. To introduce them, some notation is needed. Denote the local time of the Markov process $(\sup_{s \leq t} X_s - X_t)_{t\ge 0}$ at $0$ by $L$, which is also called the local time of $X$ at the maximum. Let $L_t^{-1} = \inf\{s>0 : L_s > t\}$ denote the inverse local time at the maximum and $\kappa(q) = -\log \mE \big[ e^{- q L^{-1}_1} \big]$, for $q \geq 0$, the 
Laplace exponent of $L^{-1}$. We define $H_t =\sup_{s\leq L^{-1}_t}X_s$, the so-called (ascending) ladder height process. It is well-known that $H$ is a subordinator. Under the dual measure $\hat\mP$, the process $L^{-1}$ is the inverse local time at the minimum, and we denote its Laplace exponent by $\hat\kappa$. Still under this dual measure, $H$ is the descending ladder height process.

\smallskip

The $q$-resolvents of $H$, for $q \geq 0$, will be denoted by $U_+^q$; that is,
$$U^q_+(\dd x) \coloneqq \mE \Big[ \int\limits_{[0,\infty)} e^{-qt} \1_{\{ H^+_t \in \dd x, L_t^{-1} <\infty\}} \, \dd t \Big].$$ 
For $q=0$ we abbreviate $U_+(\dd x) = U_+^0(\dd x)$, and denote the so-called potential function by $U_+(x) = U_+([0,x])$, for $x\geq 0$. We define $U_-^q$ and $U_-$ according to the same procedure for the descending ladder height process. 

\smallskip

In the case of a stable process all appearing fluctuation expressions are known explicitly: $\kappa(q) = q^\rho$, $\hat{\kappa}(q) = q^{1-\rho}$, $U_-(x) = x^{\alpha(1-\rho)}$ and $U_+(x) = x^{\alpha\rho}$. As it was mentioned earlier, for a stable process without negative jumps it holds $\rho = 1$ in the case $\alpha<1$ and $\rho = 1-1/\alpha$ in the case $\alpha>1$. Hence, the expressions above even simplify more.

\smallskip

\subsection{Lévy processes conditioned to stay positive}\label{subsec_cond}
For general Lévy processes (which are not compound Poisson processes) Chaumont and Doney \cite{Chau_Don_01} showed that the potential function $U_-$ of the dual ladder height process is an invariant function for the Lévy process killed on entering $(-\infty,0]$. Furthermore they showed
$$\lim_{q \searrow 0} \mP_x(\Lambda, t < e_q \,|\, e_q < T_{(-\infty,0]}) = \mE_x \Big[ \1_\Lambda \1_{\{ t < T_{(-\infty,0]}\}} \frac{U_-(X_t)}{U_-(x)} \Big],$$
i.e. the $h$-transformed process corresponding to the invariant function $U_-$ is the process conditioned to avoid $(-\infty,0]$ in the sense of \eqref{cond_exp}. Since we later consider (stable) Lévy processes without negative jumps we remark that in this case it holds $U_-(x) = x$. Analogously, one can construct the Lévy process conditioned to stay negative via the potential function of the ladder height process $U_+$. One important ingredient in the proofs are the following relations:
\begin{align}\label{as_pos}
\hat{\kappa}(q) U_-^q(x) = \mP_x (e_q < T_{(-\infty,0]})
\end{align}
and 
\begin{align}\label{as_neg}
\kappa(q) U_+^q(x) = \mP_x (e_q < T_{[0,\infty)}).
\end{align}
Because of the spatial homogeneity of Lévy processes it is possible to condition them to stay above/ below a level via an $h$-transform using shifted versions of $U_-$ and $U_+$.

\section{Proofs}\label{sec:proof}
\subsection{Proofs of Theorem \ref{mainthm1} and Corollary \ref{cor}}
First we show $h(x) = \mP_x(T_{[-1,1]} = \infty), x \notin [-1,1]$. For $x>1$ this is obvious because $X$ is a subordinator. For $x<-1$ it can be shown via overshoot-distributions for stable processes (see e.g. Rogozin \cite{Rog_01}):
\begin{align*}
\mP_x(T_{[-1,1]} = \infty) &= \mP_x(X_{T_{[-1,\infty)}} > 1) \\
&= 1 - \frac{\sin(\pi\alpha)}{\pi} \int_0^{\frac{2}{-1-x}}  t^{-\alpha} (1+t)^{-1}  \, \dd t. \\
&= h(x).
\end{align*} 
The following standard argument shows that $h$ is invariant for the process killed on entering $[-1,1]$:
\begin{align*}
\mE_x \big[ \1_{\{ t < T_{[-1,1]} \}} h(X_t) \big] &= \mE_x \big[ \1_{\{ t < T_{[-1,1]} \}} \mP_{X_t}(T_{[-1,1]} = \infty) \big]  \\
&= \lim_{s \rightarrow \infty }\mE_x \big[ \1_{\{ t < T_{[-1,1]} \}} \mP_{X_t}(T_{[-1,1]} > s) \big] \\
&= \lim_{s \rightarrow \infty } \mP_{x}(T_{[-1,1]} > s+t) \\
&= \mP_x(T_{[-1,1]} = \infty) \\
&= h(x),
\end{align*}
where we used dominated convergence and the Markov property. Furthermore, another simple calculation shows that the conditioned process in the sense of \eqref{cond} (and similar \eqref{cond_exp}) is the corresponding $h$-transformed process:
\begin{align*}
\mE_x\Big[ \1_\Lambda \1_{\{ t < T_{[-1,1]} \}}  \frac{h(X_t)}{h(x)} \Big] &= \mE_x\Big[ \1_\Lambda \1_{\{ t < T_{[-1,1]} \}}  \frac{\mP_{X_t}(T_{[-1,1]} = \infty)}{\mP_x(T_{[-1,1]} = \infty)} \Big] \\
&= \frac{1}{\mP_x(T_{[-1,1]} = \infty)}  \mP_x(\Lambda, T_{[-1,1]} = \infty)\\
&= \lim_{s \rightarrow \infty}  \frac{1}{\mP_x(T_{[-1,1]} > s)}  \mP_x(\Lambda, T_{[-1,1]} >s) \\
&= \lim_{s \rightarrow \infty}  \mP_x(\Lambda \,|\, T_{[-1,1]} >s).
\end{align*}

\subsection{Proof of Proposition \ref{propmain}}
If $X$ is a stable process with self-similarity index $\alpha \in (1,2)$ without negative jumps, it holds $U_-(x) = x$ and $U_+(x) = x^{\alpha-1}$ for $x >0$. By the results of Section \ref{subsec_cond} and the spatial homogeneity of Lévy processes we get that $x \mapsto x-1, x>1$ is an invariant function for the stable process killed on entering $(-\infty,1]$ and $x \mapsto (-x-1)^{\alpha-1}, x<-1$ is an invariant function for the stable process killed on entering $[-1,\infty)$. In particular $\mP_x^\uparrow$ (for $x>1$) and $\mP_x^\downarrow$ (for $x<-1$) are probability measures since they are constructed via $h$-transforms using invariant functions. Consequently, $\mP_x^\updownarrow$ is a probability measure for $x \notin [-1,1]$.

\smallskip

That $\mP_x^\updownarrow(X_t \in \dd y) = P_t^\updownarrow(x,\dd y)$ is obvious. It remains to show the Chapman-Kolmogorov equality for $(P_t^\updownarrow)_{t \geq 0}$. For that let us denote by $P_t^\uparrow(x,\dd y)$ the transition semigroup of the stable process conditioned to stay above $1$, i.e.
$$P_t^\uparrow(x,\dd y) = \mP^\uparrow_x(X_t \in \dd y) = \mE_x\Big[ 1_{\{X_t \in \dd y\}} \1_{\{t < T_{(-\infty,1]}\}} \frac{X_t-1}{x-1}\Big], \quad x,y>1.$$
Obviously, we have $P_t^\updownarrow(x,\dd y) = P_t^\uparrow(x,\dd y)$ for $x>1$ (we agree that $P_t^\uparrow(x,\dd y) = 0$ for $y \leq 1$). So we get for $x>1$ and $y \notin [-1,1]$:
\begin{align*}
\int_{\mR\setminus [-1,1]} P_s^\updownarrow(z, \dd y) \, P^\updownarrow_t(x,\dd z) &= \int_{\mR\setminus [-1,1]} P_s^\updownarrow(z, \dd y) \, P^\uparrow_t(x,\dd z) \\
&= \int_{(1,\infty)} P_s^\updownarrow(z, \dd y) \, P^\uparrow_t(x,\dd z) \\
&= \int_{(1,\infty)} P_s^\uparrow(z, \dd y) \, P^\uparrow_t(x,\dd z) \\
& = P^\uparrow_{t+s}(x,\dd y) \\
& = P^\updownarrow_{t+s}(x,\dd y).
\end{align*}
The second last equality holds because $(P_t^\uparrow)_{t \geq 0}$ fulfils the Chapman-Kolmogorov equality since the Lévy process conditioned to stay positive is a Markov process. The last equality holds because $P^\updownarrow_{t+s}(x,\dd y)= 0 = P^\uparrow_{t+s}(x,\dd y)$ for $x>1,y<-1$. The case $x<-1$ can be handled analogously.

\subsection{Proof of Theorem \ref{mainthm}}
With the knowledge of stable processes conditioned to stay positive the case $x>1$ is almost trivial. We remind ourselves that due to no negative jumps, the process can not reach the area below the interval without hitting it, i.e. it holds $T_{(-\infty,1]} = T_{[-1,1]}$ a.s. under $\mP_x$. Hence,
\begin{align*}
\lim_{q \searrow 0} \mP_x(\Lambda, t < e_q \,|\, e_q < T_{[-1,1]}) &= \lim_{q \searrow 0} \mP_x(\Lambda, t < e_q \,|\, e_q < T_{(-\infty,1]}) \\
&= \mE_x \Big[\1_\Lambda \1_{\{t < T_{(-\infty,1]}\}} \frac{U_-(X_t-1)}{U_-(x-1)}\Big]  \\
&= \mE_x \Big[\1_\Lambda \1_{\{t < T_{(-\infty,1]}\}} \frac{X_t-1}{x-1}\Big] \\
&= \mP_x^\updownarrow(\Lambda).
\end{align*}

\smallskip

From now on let $x<-1$. We follow the usual approach and consider first the asymptotics of $\mP_x(e_q < T_{[-1,1]})$ for $q \searrow 0$.
\begin{lemma}\label{lemma_help}
There is a constant $c>0$ such that 
$$\lim_{q \searrow 0} q^{\frac{1}{\alpha}-1} \mP_x(e_q < T_{[-1,1]}) = c (-1-x)^{\alpha-1}, \quad x <-1.$$
\end{lemma}
\begin{proof}
We apply a Theorem of Port \cite{Por_01} which says that there is a constant $\tilde c >0$ such that
\begin{align}\label{help1}
\lim_{q \searrow 0} s^{1-\frac{1}{\alpha}} \mP_x(s < T_{[-1,1]}) = \tilde c \lim_{y \rightarrow -\infty}u^{[-1,1]}(x,y), \quad x <-1,
\end{align}
where for a general compact set $y \mapsto u^{K}(x,y)$ is the potential density of the stable process killed on entering $K$ started in $x$. Using that the process has no negative jumps and Bertoin \cite{Bert_01}, Proposition VI.20 leads to
\begin{align}\label{help2}
\begin{split}
u^{[-1,1]}(x,y) &= u^{[-1,\infty)}(x,y) \\
&= \frac{k}{\Gamma(\alpha)} \big((-x-1)^{\alpha-1} - (-x+y)_+^{\alpha-1} \big),
\end{split}
\end{align}
for $x,y<-1$, where $k>0$ is a constant. Combining \eqref{help1} and \eqref{help2} we get 
\begin{align}
\lim_{q \searrow 0} s^{1-\frac{1}{\alpha}} \mP_x(s < T_{[-1,1]}) = \hat c (-x-1)^{\alpha-1}, \quad x <-1,
\end{align} 
where $\hat c = \tilde c k/ \Gamma(\alpha)$. To reach the asymptotic including the exponential time, we note,
\begin{align}
\mP_x(e_q < T_{[-1,1]}) = q \int_0^\infty \mathrm{e}^{-qs} \mP_x(s < T_{[-1,1]}) \, \dd s.
\end{align}
Since we know that $\mP_x(s < T_{[-1,1]})$ behaves like $\hat c (-x-1)^{\alpha-1} s^{\frac{1}{\alpha}-1}$ for $s$ tending to $\infty$, we can apply standard Tauberian Theorems (see e.g. \cite{Kyp_01}, Theorem 5.14) to deduce that $\mP_x(e_q < T_{[-1,1]})$ behaves like
$$\frac{\hat c}{\Gamma(\frac{1}{\alpha})} (-x-1)^{\alpha-1} q^{1-\frac{1}{\alpha}} = c (-x-1)^{\alpha-1} q^{1-\frac{1}{\alpha}}$$
for $q \searrow 0$ which shows the claim.
\end{proof}

With this lemma in hand we are ready to show Theorem \ref{mainthm} for the remaining case $x<-1$.
\begin{proof}[Proof of Theorem \ref{mainthm}]
Let $x<-1$. First, we note
\begin{align}\label{help4}
\begin{split}
\mP_x(\Lambda, t < e_q < T_{[-1,1]}) &= q \int_0^\infty \mathrm{e}^{-qu} \mP_x(\Lambda, t < u < T_{[-1,1]}) \, \dd u \\
&= q \int_t^\infty  \mathrm{e}^{-qu} \mP_x(\Lambda,  u < T_{[-1,1]}) \, \dd u \\
&= q \mathrm{e}^{-qt} \int_0^\infty  \mathrm{e}^{-qu} \mP_x(\Lambda,  u+t < T_{[-1,1]}) \, \dd u \\
&= q \mathrm{e}^{-qt} \int_0^\infty  \mathrm{e}^{-qu} \mE_x \big[\1_\Lambda \1_{\{ t <T_{[-1,1]} \}} \mP_{X_t}(u < T_{[-1,1]}) \big]  \, \dd u \\
&= \mathrm{e}^{-qt} \mE_x \big[\1_\Lambda \1_{\{ t <T_{[-1,1]} \}} \mP_{X_t}(e_q < T_{[-1,1]}) \big].
\end{split}
\end{align} 
Now we start on the left-hand-side of \eqref{eq_main} and do some basic manipulations like applying \eqref{help4} and the Markov property.
\begin{align}\label{help3}
\begin{split}
& \mP_x(\Lambda, t < e_q \,|\, e_q < T_{[-1,1]})\\
&=\frac{\mP_x(\Lambda, t < e_q < T_{[-1,1]})}{\mP_x(e_q < T_{[-1,1]})} \\
&=\frac{\mathrm{e}^{-qt}}{\mP_x(e_q < T_{[-1,1]})} \mE_x \big[\1_\Lambda \1_{\{ t <T_{[-1,1]} \}} \mP_{X_t}(e_q < T_{[-1,1]}) \big] \\
&= \frac{\mathrm{e}^{-qt}}{\mP_x(e_q < T_{[-1,1]})} \mE_x \big[\1_\Lambda \1_{\{ t <T_{[-1,\infty)} \}} \mP_{X_t}(e_q < T_{[-1,1]}) \big] \\
& \quad + \frac{\mathrm{e}^{-qt}}{\mP_x(e_q < T_{[-1,1]})} \mE_x \big[\1_\Lambda \1_{\{ t \in [T_{[-1,\infty)},T_{[-1,1]}) \}} \mP_{X_t}(e_q < T_{[-1,1]}) \big] \\
&= \frac{\mathrm{e}^{-qt}}{\mP_x(e_q < T_{[-1,1]})} \mE_x \big[\1_\Lambda \1_{\{ t <T_{[-1,\infty)} \}} \mP_{X_t}(e_q < T_{[-1,1]}) \big] \\
& \quad + \frac{\mathrm{e}^{-qt}}{\mP_x(e_q < T_{[-1,1]})} \mE_x \big[\1_\Lambda \1_{\{ t \in [T_{[-1,\infty)},T_{[-1,1]}) \}} \mP_{X_t}(e_q < T_{(-\infty,1]}) \big].
\end{split}
\end{align}
In the next step we show that the second summand of the last term converges to $0$ for $q \searrow 0$. From \eqref{as_pos} and the explicit form $U_-(x) = x$ we know
$$\mP_{X_t}(e_q < T_{(-\infty,1]}) = q^{\frac{1}{\alpha}} U_-^q(X_t-1) \leq q^{\frac{1}{\alpha}} U_-(X_t-1) = q^{\frac{1}{\alpha}} (X_t-1).$$
In particular it holds
\begin{align*}
&  \frac{\mathrm{e}^{-qt}}{\mP_x(e_q < T_{[-1,1]})} \mE_x \big[\1_\Lambda \1_{\{ t \in [T_{[-1,\infty)},T_{[-1,1]}) \}} \mP_{X_t}(e_q < T_{(-\infty,1]}) \big] \\
&\leq \frac{\mathrm{e}^{-qt}}{\mP_x(e_q < T_{[-1,1]})} q^{\frac{1}{\alpha}} \mE_x \big[\1_\Lambda \1_{\{ t \in [T_{[-1,\infty)},T_{[-1,1]}) \}} (X_t-1)\big].
\end{align*}
Since an $\alpha$-stable process has all moments of order less than $\alpha$ it holds 
$$\mE_x \big[\1_\Lambda \1_{\{ t \in [T_{[-1,\infty)},T_{[-1,1]}) \}} (X_t-1)\big] < \infty.$$
On the other hand by Lemma \ref{lemma_help}, $q^{\frac{1}{\alpha}} / \mP_x(e_q < T_{[-1,1]})$ behaves like $q^{\frac{1}{\alpha}}/(c q^{1-\frac{1}{\alpha}}) = q^{\frac{2}{\alpha}-1}/c$ which converges to $0$ because $\alpha<2$ implies $2/\alpha >1$. By this we obtain 
$$\lim_{q \searrow 0}  \frac{\mathrm{e}^{-qt}}{\mP_x(e_q < T_{[-1,1]})} \mE_x \big[\1_\Lambda \1_{\{ t \in [T_{[-1,\infty)},T_{[-1,1]}) \}} \mP_{X_t}(e_q < T_{(-\infty,1]}) \big] = 0.$$
Now we plug this in \eqref{help3} to calculate the limit for $q \searrow 0$. We use Lemma \ref{lemma_help} and a standard procedure in the area of Lévy processes conditioned to avoid a set. First we get via Fatou's lemma
\begin{align*}
& \liminf_{q \searrow 0} \mP_x(\Lambda, t < e_q \,|\, e_q < T_{[-1,1]})\\
&= \liminf_{q \searrow 0} \frac{\mathrm{e}^{-qt}}{\mP_x(e_q < T_{[-1,1]})} \mE_x \big[\1_\Lambda \1_{\{ t <T_{[-1,\infty)} \}} \mP_{X_t}(e_q < T_{[-1,1]}) \big] \\
&= \frac{1}{c(-1-x)^{\alpha-1}} \liminf_{q \searrow 0} \mE_x \big[\1_\Lambda \1_{\{ t <T_{[-1,\infty)} \}} q^{\frac{1}{\alpha}-1}\mP_{X_t}(e_q < T_{[-1,1]}) \big] \\
& \geq  \frac{1}{c(-1-x)^{\alpha-1}} \mE_x \big[\1_\Lambda \1_{\{ t <T_{[-1,\infty)} \}} \liminf_{q \searrow 0} q^{\frac{1}{\alpha}-1}\mP_{X_t}(e_q < T_{[-1,1]}) \big] \\
&= \frac{1}{c(-1-x)^{\alpha-1}}  \mE_x \big[\1_\Lambda \1_{\{ t <T_{[-1,\infty)} \}} c (-1-X_t)^{\alpha-1} \big] \\
&= \mP_x^\updownarrow(\Lambda).
\end{align*}
Next, note that for $\Lambda \in \cF_t$ also $\Lambda^{\mathrm{C}} \in \cF_t$. Moreover, we already know that $\mP^\updownarrow_x$ is a probability measure by Proposition \ref{propmain}. It follows
\begin{align*}
\limsup_{q \searrow 0} \mP_x(\Lambda, t < e_q \,|\, e_q < T_{[-1,1]}) &\leq \limsup_{q \searrow 0} (1-\mP_x(\Lambda^\mathrm{C}, t < e_q \,|\, e_q < T_{[-1,1]})) \\
&= 1 - \liminf_{q \searrow 0}\mP_x(\Lambda^\mathrm{C}, t < e_q \,|\, e_q < T_{[-1,1]}) \\
& \leq 1-  \mP_x^\updownarrow(\Lambda^\mathrm{C}) \\
&=  \mP_x^\updownarrow(\Lambda).
\end{align*}
Together this shows
$$\lim_{q \searrow 0} \mP_x(\Lambda, t < e_q \,|\, e_q < T_{[-1,1]}) = \mP_x^\updownarrow(\Lambda).$$
\end{proof}

\bibliographystyle{abbrvnat}
\bibliography{references}

\end{document}